\documentclass[11pt, letterpaper]{article}
\usepackage{amssymb, url}
\usepackage{amsthm}
\usepackage{amsmath}
\usepackage{verbatim}
\usepackage{graphicx}
\usepackage{caption}
\usepackage{epstopdf}
\usepackage{fullpage}
\usepackage{setspace}
\usepackage{enumerate}
\usepackage{float}
\usepackage{tikz}
\usepackage{multicol}
\usepackage{hyperref}
\usetikzlibrary{shapes,arrows,spy,positioning,decorations,calc}
\usepackage{subfigure}
\usepackage[normalem]{ulem}
\usepackage[ruled,vlined,linesnumbered,nofillcomment,noend]{algorithm2e}
\usepackage[hang,flushmargin]{footmisc}

\newcounter{results}[section]

\newtheorem{thm}[results]{Theorem}
\newtheorem{cor}[results]{Corollary}
\newtheorem{lem}[results]{Lemma}
\newtheorem{prop}[results]{Proposition}
\newtheorem{conj}[results]{Conjecture}

\newtheorem{clm}[results]{Claim}

\newcommand{\sci}{\chi_{s}'(G)}
\newcommand{\scil}{\chi_{\ell,s}'(G)}
\DeclareMathOperator{\ex}{ex}
\DeclareMathOperator{\girth}{girth}
\DeclareMathOperator{\mad}{mad}
\DeclareMathOperator{\ct}{ct}
\DeclareMathOperator{\Resp}{Resp}

\tikzset{
insep/.style={inner sep=2pt, outer sep=0pt, circle,fill}, 
free/.style={inner sep=2pt, outer sep=0pt, circle,fill=gray,draw}, 
extra/.style={inner sep=2pt, outer sep=0pt, circle,fill=white,draw}, 
}

\newcommand{\bla}{{\ensuremath{^\perp}}}

\newcounter{casenum}	
\newcounter{subcasenum}	
\numberwithin{subcasenum}{casenum}
\newcounter{subsubcasenum}	
\numberwithin{subsubcasenum}{subcasenum}

\renewcommand{\thecasenum}{\arabic{casenum}}
\renewcommand{\thesubcasenum}{\thecasenum.\roman{subcasenum}}

\newcounter{stagenum}

\newenvironment{mycases}
{
  \list{}{%
    \leftmargin0.5cm   
    \rightmargin0cm
  }
  \item\relax
	\setcounter{casenum}{0}
}
{	
	\endlist
}
\newenvironment{subcases}
{
  \list{}{%
    \leftmargin0.5cm   
    \rightmargin0cm
  }
  \item\relax
}
{	
	\endlist
}

\newcommand{\mycase}[1]{
	\vspace{0.5em}
	
	\refstepcounter{casenum}
	\noindent\hspace{-0.5cm}\textit{Case \thecasenum: #1} 
}

\newcommand{\subcase}[1]{
	\vspace{0.25em}
	
	\refstepcounter{subcasenum}
	\noindent\hspace{-0.5cm}\textit{Case \thesubcasenum: #1} 
}

\begin{document}

\title{On the Strong Chromatic Index of Sparse Graphs}
\author{ Philip DeOrsey$^{1,6}$ \and
	Jennifer Diemunsch$^{2,6}$ \and 
	Michael Ferrara$^{2,6,7}$ \and 
	Nathan Graber$^{2,6}$ \and
	Stephen G. Hartke$^{3,6,8}$ \and
	Sogol Jahanbekam$^{2,6}$ \and
	Bernard Lidick\'y$^{4,6,9}$ \and
	Luke Nelsen$^{2,6}$ \and
	Derrick Stolee$^{4,5,6}$ \and
	Eric Sullivan$^{2,6}$}
\maketitle

\footnotetext[1]{Department of Mathematics, Emory and Henry College, Emory, VA 24327; {\tt pdeorsey@ehc.edu}.}
\footnotetext[2]{Department of Mathematical and Statistical Sciences, University of Colorado Denver, Denver, CO 80217; {\tt $\{$jennifer.diemunsch,michael.ferrara,nathan.graber,sogol.jahanbekam,luke.nelsen,eric.2.sullivan$\}$@ucdenver.edu}.}
\footnotetext[3]{Department of Mathematics, University of Nebraska--Lincoln, Lincoln, NE 68588; {\tt hartke@math.unl.edu}}
\footnotetext[4]{Department of Mathematics, Iowa State University, Ames, IA 50011; {\tt $\{$lidicky,dstolee$\}$@iastate.edu}}
\footnotetext[5]{Department of Computer Science, Iowa State University, Ames, IA 50011.}
\footnotetext[6]{Research supported in part by NSF grants DMS-1427526, ``The Rocky Mountain - Great Plains Graduate Research Workshop in Combinatorics" and DMS-1500662 ``The 2015 Rocky Mountain - Great Plains Graduate Research Workshop in Combinatorics". For more information about the GRWC, please see {\tt https://sites.google.com/site/rmgpgrwc}.}  
\footnotetext[7]{Research supported in part by a Collaboration Grant from the Simons Foundation (\#206692 to Michael Ferrara).}
\footnotetext[8]{Research supported in part by a Collaboration Grant from the Simons Foundation (\#316262 to Stephen G. Hartke).}
\footnotetext[9]{Research supported in part by NSF grant DMS-126601.}  

\begin{abstract}
The strong chromatic index of a graph $G$, denoted $\sci$, is the least number of colors needed to edge-color $G$ so that edges at distance at most two receive distinct colors. 
The strong list chromatic index, denoted $\scil$, is the least integer $k$ such that if arbitrary lists of size $k$ are assigned to each edge then $G$ can be edge-colored from those lists where edges at distance at most two receive distinct colors. 
We use the discharging method, the Combinatorial Nullstellensatz, and computation to show that if $G$ is a subcubic planar graph with $\girth(G) \geq 41$ then $\scil \leq 5$, answering a question of Borodin and Ivanova [Precise upper bound for the strong edge chromatic number of sparse planar graphs, \textit{Discuss. Math. Graph Theory}, 33(4), (2014) 759--770]. We further show that if $G$ is a subcubic planar graph and  $\girth(G) \geq 30$, then $\sci \leq 5$, improving a bound from the same paper.
Finally, if $G$ is a planar graph with maximum degree at most four and $\girth(G) \geq 28$, then $\sci \leq 7$, improving a more general bound of Wang and Zhao from [Odd graphs and its application on the strong edge coloring, \texttt{arXiv:1412.8358}] in this case.
\end{abstract}

\section{Introduction}

A {\it proper edge-coloring} of a graph $G$ is an assignment of colors to the edges so that incident edges receive distinct colors. 
A {\it strong edge-coloring} of a graph $G$ is an assignment of colors to the edges so that edges at distance at most two receive distinct colors. 
A proper edge-coloring is a decomposition of $G$ into matchings, while a strong edge-coloring is a decomposition of $G$ into \emph{induced} matchings.
Fouquet and Jolivet~\cite{Fouquet83,Fouquet84} defined the {\it strong chromatic index} of a graph $G$, denoted $\sci$, as the minimum integer $k$ such that $G$ has a strong edge-coloring using $k$ colors.
Erd\H{o}s and Ne\v{s}et\v{r}il gave the following conjecture, which is still open, and provided an example to show that it would be sharp, if true.  

\begin{conj}[Erd\H{o}s and Ne\v{s}et\v{r}il \cite{EN1988}]\label{Erdos}
For every graph $G$, $\chi'_s(G) \leq \dfrac{5}{4}\Delta(G)^2$ when $\Delta(G)$ is even, and $\chi'_s(G) \leq \dfrac{1}{4}(5\Delta(G)^2-2\Delta(G)+1)$ when $\Delta(G)$ is odd.
\end{conj}

Towards this conjecture, Molloy and Reed~\cite{MR}  bounded $\sci$ away from the trivial upper bound of $2\Delta(G)(\Delta(G)-1)+1$ by showing that every graph $G$ with sufficiently large maximum degree satisfies $\sci\leq1.998\Delta(G)^2$.
Bruhn and Joos~\cite{BJ} have announced an improvement, claiming $\sci \leq 1.93\Delta(G)^2$.


The focus of this paper is the study of strong edge-colorings of \emph{subcubic graphs}, those with maximum degree at most three, and \emph{subquartic graphs}, those with maximum degree at most four.
Faudree, Gy\'arfas, Schelp, and Tuza~\cite{Faudree90} studied $\sci$ in the class of subcubic graphs, and gave the following conjectures.

\begin{conj}[Faudree \textit{et al.} \cite{Faudree90}]
Let $G$ be a subcubic graph.

\begin{enumerate}[(1)]\setlength{\itemsep}{0em}
\item $\sci \leq 10$.
\item If $G$ is bipartite, then $\sci \leq 9$. \label{con:bipartite}
\item If $G$ is planar, then $\sci \leq 9$. \label{con:planar}
\item If $G$ is bipartite and for each edge $xy \in E(G)$, $d(x) + d(y) \leq 5$, then $\sci \leq 6$.
\item If $G$ is bipartite and $C_4 \not \subset G$, then $\sci \leq 7$.
\item If $G$ is bipartite and its girth is large, then $\sci \leq 5$. \label{con:girth}
\end{enumerate}
\label{con:Faud}
\end{conj}

Several of these conjectures have been verified, including (1) by Andersen~\cite{Andersen} and (2) by Steger and Yu~\cite{StegerYu}.  Quite recently, Kostochka, Li, Ruksasakchai, Santana, Wang, and Yu~\cite{KLRSWY2014} announced an affirmative resolution to (3).
This result is best possible since the prism, shown in Figure \ref{prism}, is a subcubic planar graph with $\sci = 9$.

\begin{figure}[h]
\centering
\begin{tikzpicture}[scale=.6]

\draw (0,0)--(9,0)
(0,0)--(0,4)
(0,4)--(9,4)
(9,4)--(9,0);

\draw (0,0) node[insep]{}
(0,4) node[insep]{}
(9,4) node[insep]{}
(9,0) node[insep]{}
(2,2) node[insep]{}
(7,2) node[insep]{};

\draw (0,0)--(2,2)
(0,4)--(2,2)
(2,2)--(7,2)
(7,2)--(9,4)
(7,2)--(9,0);

\end{tikzpicture}

	\caption{The prism is a subcubic planar graph $G$ with $\sci = 9$.}
	\label{prism}
\end{figure}
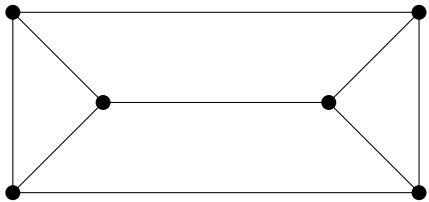

Several papers prove sharper bounds on the strong chromatic index of planar graphs with additional structure~\cite{Fouquet84,Hocquard11,Hocquard13,Hudak}, generally by introducing conditions on maximum average degree or girth to ensure that the target graph is sufficiently sparse.  
For a graph $G$, the {\it maximum average degree} of $G$, denoted $\mad(G)$, is the maximum of average degrees over all subgraphs of $G$.
Hocquard, Montassier,  Raspaud, and Valicov~\cite{Hocquard11, Hocquard13} proved the following.

\clearpage
\begin{thm}[Hocquard \textit{et al.} \cite{Hocquard13}]\label{thm:hocquard}
Let $G$ be a subcubic graph.
\begin{enumerate}
\item If $\mad(G) < \frac{7}{3}$, then $\sci \leq 6$.  \label{part1}
\item If $\mad(G) < \frac{5}{2}$, then $\sci \leq 7$. \label{part2}
\item If $\mad(G) < \frac{8}{3}$, then $\sci \leq 8$.
\end{enumerate}
\label{thm:Hoc}
\end{thm}

\begin{figure}
\centering
\begin{minipage}{.5\textwidth}
  \centering
  \begin{tikzpicture}[scale=1]

\draw (0,0)--(0,2)
(0,2)--(2,2)
(2,2)--(2,0)
(2,0)--(0,0);

\draw
(0,0) node[insep]{}
(0,2) node[insep]{}
(2,2) node[insep]{}
(2,0) node[insep]{}
(1,2) node[insep]{}
(1,3.5) node[insep]{}
(0,2)--(1,3.5)
(1,3.5)--(2,2);

\end{tikzpicture}
  \captionof{figure}{A graph $G$ with $\mad(G) = \frac{7}{3}$ \\ and $\sci > 6$.}
  \label{fig:house}
\end{minipage}%
\begin{minipage}{.5\textwidth}
  \centering
  \begin{tikzpicture}[scale=1]

\draw (-1.5,0)--(1.5,0)
(-1.5,0) node[insep]{}
(-.5,0) node[insep]{}
(.5,0) node[insep]{}
(1.5,0) node[insep]{}
(90:1.5) node[insep]{}
(270:1.5) node[insep]{}
(-1.5,0)--(90:1.5)
(90:1.5)--(1.5,0)
(1.5,0)--(270:1.5)
(270:1.5)--(-1.5,0)
(90:1.5)--(2.5,1.5)
(2.5,1.5)--(2.5,-1.5)
(270:1.5)--(2.5,-1.5)
(2.5,1.5) node[insep]{}
(2.5,-1.5) node[insep]{}
;

\end{tikzpicture}
  \captionof{figure}{A graph $G$ with $\mad(G) = \frac{5}{2}$ \\ and $\sci > 7$.}
  \label{fig:diamond}
\end{minipage}
\end{figure}

Parts (\ref{part1}) and (\ref{part2}) of Theorem~\ref{thm:hocquard} are sharp by the graphs shown in Figures \ref{fig:house} and \ref{fig:diamond}, respectively.
An elementary application of Euler's Formula (see \cite{west}) gives the following.

\begin{prop}
If $G$ is a planar graph with girth $g$ then $\mad(G) < \frac{2g}{g-2}$.
\label{prop:mad}
\end{prop}
Theorem \ref{thm:Hoc} and Proposition \ref{prop:mad} yield the following corollary.

\begin{cor}[Hocquard \textit{et al.} \cite{Hocquard13}]
Let $G$ be a subcubic planar graph with girth $g$.

\begin{enumerate}
\item If $g \geq 14$, then $\sci \leq 6$.
\item If $g \geq 10$, then $\sci \leq 7$.
\item If $g \geq 8$, then $\sci \leq 8$.
\end{enumerate}
\end{cor}

Note that no non-trivial sparsity condition on a graph $G$ with maximum degree $d$ will guarantee that $\sci < 2d - 1$ since any graph having two adjacent vertices of degree $d$ requires at least $2d-1$ colors to strongly edge-color the graph.
We give sparsity conditions that imply a subcubic planar graph has strong chromatic index at most five and a subquartic planar graph has strong chromatic index at most seven.
Previous work in this direction was initiated by Borodin and Ivanova~\cite{BI}, Chang, Montassier, P\v{e}cher, and Raspaud~\cite{CMPR}, and most recently extended by Wang and Zhao~\cite{WZ}.
The current-best bounds are given by the following two results.

\begin{thm}[Borodin and Ivanova~\cite{BI}]\label{thm:bi}
Let $G$ be a subcubic graph.
\begin{enumerate}
\item If $G$ has girth at least $9$ and $\mad(G) < 2 + \frac{2}{23}$, then $\chi_s'(G) \leq 5$.
\item If $G$ is planar and has girth at least $41$, then $\chi_s'(G) \leq 5$.
\end{enumerate}
\end{thm}

\begin{thm}[Wang and Zhao~\cite{WZ}]\label{thm:wz}
Fix $d \geq 4$ and let $G$ be a graph with $\Delta(G) \leq d$.
\begin{enumerate}
\item If $G$ has girth at least $2d-1$ and $\mad(G) < 2 + \frac{2}{6d-7}$, then $\chi_s'(G) \leq 2d-1$.
\item If $G$ is planar and has girth at least $10d-4$, then $\chi_s'(G) \leq 2d-1$.
\end{enumerate}
\end{thm}


%
%
%
%
%

One barrier to proving sparsity conditions that imply $\sci \leq 5$ is that there exist graphs $G$ with $\mad(G) = 2$ and $\sci = 6$.
Let $S_3$ be a triangle with pendant edges at each vertex, and let $S_4$ be a $4$-cycle with pendant edges at two adjacent vertices. For $k \geq 5$, let $S_k$ be a $k$-cycle with pendant edges at each vertex.
Each of $S_3$, $S_4$ and $S_7$ have maximum average degree $2$ and strong chromatic index at least 6, see Figure~\ref{fig:S3S4S5}. However, these graphs are 6-critical with respect to $\sci$, as the removal of any edge from $S_3$, $S_4$ or $S_7$ results in a graph that has a strong edge-coloring using five colors. 

\begin{figure}[h]
\centering
{\usetikzlibrary[topaths]

\tikzset{
insep/.style={inner sep=2pt, outer sep=0pt, circle,fill}, 
free/.style={inner sep=2pt, outer sep=0pt, circle,fill=gray,draw}, 
extra/.style={inner sep=2pt, outer sep=0pt, circle,fill=white,draw}, 
}


\begin{tikzpicture}[scale=1]
\begin{scope}
\draw
\foreach \x in {0,1,2} { (0,0) ++ (90+\x*120:1) node[insep] (x\x){} -- ++(90+\x*120:0.8) node[insep]{} }
(x0)--(x1)--(x2)--(x0)
;
\draw(0,0) ++(0,-2) node{$S_3$};
\end{scope}

\begin{scope}[xshift=4cm]
\draw
\foreach \x in {0,1,2,3} { (0,0) ++ (45+\x*90:1) node[insep] (x\x){}  }
(x0)--(x1)--(x2)--(x3)--(x0)
(x0) -- ++(45+0*90:0.8) node[insep]{}
(x3) -- ++(45+3*90:0.8) node[insep]{}
;
\draw(0,0) ++(0,-2) node{$S_4$};
\end{scope}

\begin{scope}[xshift=8cm]
\draw
\foreach \x in {0,1,...,6} { (0,0) ++ (90+\x*360/7:1) node[insep] (x\x){} --  ++ (90+\x*360/7:0.8) node[insep]{}  }
(x0)--(x1)--(x2)--(x3)--(x4)--(x5)--(x6)--(x0)
;
\draw(0,0) ++(0,-2) node{$S_7$};
\end{scope}
\end{tikzpicture}


}
\caption{Exceptions in Theorem~\ref{thm:sparse}.\label{fig:S3S4S5}}
\end{figure}
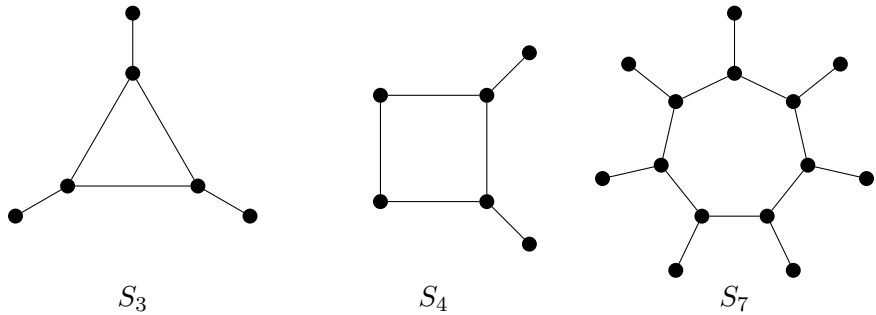

Our main theorem demonstrates that if these few graphs are avoided, and the maximum average degree is not too large, then we can find a strong 5-edge-coloring, improving Theorem~\ref{thm:bi}.

\begin{thm}\label{thm:sparse}
Let $G$ be a subcubic graph.
\begin{enumerate}
\item If $G$ does not contain $S_3$, $S_4$, or $S_7$ and $\mad(G) < 2 + \frac{1}{7}$, then $\sci \leq 5$.
\item If $G$ is planar and has girth at least $30$, then $\sci \leq 5$.
\end{enumerate}
\end{thm}

The bound in Theorem~\ref{thm:sparse} is likely not sharp, but is close to optimal.
The graph in Figure~\ref{fig:thetaexample} is subcubic, avoids $S_3$, $S_4$, and $S_7$, and satisfies both $\sci = 6$ and $\mad(G) = 2 + \frac{1}{6}$.

\begin{figure}[htp]
\centering
\tikzset{
insep/.style={inner sep=2pt, outer sep=0pt, circle,fill}, 
}

\begin{tikzpicture}[scale=1]
\draw
(0,0) node[insep](a){} 
(5,0) node[insep](b){}

(a) to[bend left=50] node[insep,pos=0.25](x){} node[insep,pos=0.50](y){} node[insep,pos=0.75](z){}  (b)
(x)-- ++(90:0.5)node[insep]{}
(y)-- ++(90:0.5)node[insep]{}
(z)-- ++(90:0.5)node[insep]{}

(a) to[bend right=50] node[insep,pos=0.25](x){} node[insep,pos=0.50](y){} node[insep,pos=0.75](z){}  (b)
(x)-- ++(-90:0.5)node[insep]{}
(y)-- ++(-90:0.5)node[insep]{}
(z)-- ++(-90:0.5)node[insep]{}

(a) to node[insep,pos=0.20](x){} node[insep,pos=0.40](y){} node[insep,pos=0.60](z){} node[insep,pos=0.80](w){}  (b)
(x)-- ++(90:0.5)node[insep]{}
(y)-- ++(90:0.5)node[insep]{}
(z)-- ++(90:0.5)node[insep]{}
(w)-- ++(90:0.5)node[insep]{}
;

\draw
;
\end{tikzpicture}

\caption{\label{fig:thetaexample}The graph $G = \ex_3(\Theta_{4,5,4})$ with $\mad(G) = 2 + \frac{1}{6}$ and $\sci = 6$.} 
\end{figure}
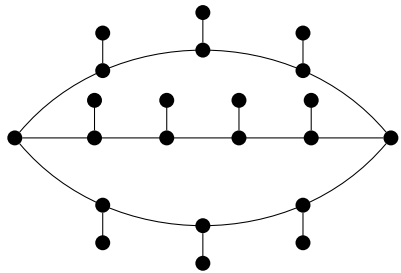


Using similar methods, we improve the bounds in Theorem~\ref{thm:wz} when $d = 4$.

\begin{thm}\label{thm:sparse4}
Let $G$ be a subquartic graph. 
\begin{enumerate}
\item If $G$ has girth at least $7$ and  $\mad(G) < 2 + \frac{2}{13}$, then $\chi_s'(G) \leq 7$.
\item If $G$ is planar and has girth at least $28$, then $\chi_s'(G) \leq 7$.
\end{enumerate}
\end{thm}

We also consider a list variation of the strong chromatic index of $G$, first introduced by Vu~\cite{vu}. 
A {\it strong list edge-coloring} of a graph $G$ is an assignment of lists to $E(G)$ such that a strong edge-coloring can be chosen from the lists at each edge. 
The minimum $k$ such that a graph $G$ can be strongly list edge-colored using any lists of size at least $k$ on each edge is the {\it strong list chromatic index} of $G$, denoted $\scil$. 
Borodin and Ivanova~\cite{BI} asked if there are sparsity conditions that imply $\scil \leq 2d-1$ for a planar graph $G$ with maximum degree $d$.
We generalize the bounds in Theorem~\ref{thm:bi} to apply to list coloring.

\begin{thm}\label{thm:mainlist}
Let $G$ be a subcubic graph.
\begin{enumerate}
\item If $G$ has girth at least $9$ and $\mad(G) < 2 + \frac{2}{23}$, then $\scil \leq 5$.
\item If $G$ is planar and has girth at least $41$, then $\scil \leq 5$.
\end{enumerate}
\end{thm}

The proofs of Theorems~\ref{thm:sparse}, \ref{thm:sparse4}, and \ref{thm:mainlist} use the discharging method.
We begin by proving Theorem~\ref{thm:mainlist} in Section~\ref{sec:list} as the proof is shorter and the one reducible configuration is used again in the proof of Theorem~\ref{thm:sparse} in Section~\ref{sec:sparse}.

\subsection{Preliminaries and Notation}

Throughout this paper we will only consider simple, finite, undirected graphs. 
We refer to \cite{west} for any undefined definitions and notation.
A graph $G$ has vertex set $V(G)$, edge set $E(G)$, and maximum degree $\Delta(G)$.

If a vertex $v$ has degree $j$ we refer to it as a {\it $j$-vertex}, and if $v$ has a neighbor that is a $j$-vertex, we say it is a {\it $j$-neighbor} of $v$. When $G$ is planar we let $F(G)$ denote the set of faces of $G$, and $\ell(f)$ denote the length of a face $f$. The {\it girth} of a graph $G$ is length of its shortest cycle.
A graph $G$ is $\{a,b\}$-regular if for every $v$ in $G$, the degree of $v$ is either $a$ or $b$. 
Every graph $G$ with maximum degree $d$ is contained in a prescribed $\{1,d\}$-regular graph, denoted $\ex_d(G)$, the {\it $d$-expansion} of $G$. 
To construct $\ex_d(G)$, add $d-d(v)$ pendant edges to each vertex $v$ in $G$ where $d(v) \in \{2,\dots, d\}$. 
Additionally, let the {\it contracted graph} of $G$, denoted $\ct(G)$ be the graph obtained by deleting all 1-vertices of $G$. 
A vertex $v$ in $G$ is a \emph{$2\bla$-vertex} if $v$ is a 2-vertex in $\ct(G)$.
Thus, for the remainder of the paper a vertex $v$ is a \emph{$k^+$-vertex} in $G$ if it has degree at least $k$ in $\ct(G)$.

We will make use of the discharging method for some of our results.
For an introduction to this method, see the survey by Cranston and West~\cite{CW}.
We will directly use two standard results that can be proven using this method.
Both of Theorems~\ref{thm:bi} and \ref{thm:wz} rely on Lemmas~\ref{lma:cw} and \ref{lma:nrs}.

Let $G$ be a graph and $\ct(G$) be the contracted graph.
An \emph{$\ell$-thread} is a path $v_1\dots v_\ell$ in $\ct(G)$ where each $v_i$ is a $2\bla$-vertex.

\begin{lem}[Cranston and West~\cite{CW}]\label{lma:cw}
If $G$ is a graph with girth at least $\ell+1$ and $\mad(G) < 2 + \frac{2}{3\ell - 1}$, then $\ct(G)$ contains a 1-vertex or an $\ell$-thread.
\end{lem}

\begin{lem}[Ne\v{s}et\v{r}il, Raspaud, and Sopena~\cite{NRS}]\label{lma:nrs}
If $G$ is a planar graph with girth at least $5\ell + 1$, then $\ct(G)$ contains a 1-vertex or an $\ell$-thread.
\end{lem}

\section{Strong List Edge-Coloring of Subcubic Graphs}\label{sec:list}

In this section, we prove Theorem~\ref{thm:mainlist}.
Our proof uses the discharging method, wherein we assign an initial charge to the vertices and faces of a theoretical minimal counterexample. This initial charge is then disbursed according to a set of discharging rules in order to draw a contradiction to the existence of such a minimal counterexample. 
We will often make use of the following, which is another simple and well known application of Euler's Formula.

\begin{prop}
\label{prop:sum}
In a planar graph $G$,

$$\sum_{f \in F(G)} ( \ell(f) - 6 ) + \sum_{v \in V(G)} (2d(v) - 6) = -12.$$

\end{prop}

We will also use the Combinatorial Nullstellensatz, which will be applied to show we can extend certain list colorings. 

\begin{thm}[Combinatorial Nullstellensatz \cite{Alon95}]\label{CN}
Let $f$ be a polynomial of degree $t$ in $m$ variables over a field $\mathbb{F}$. If there is a monomial $\prod x_i^{t_i}$ in $f$ with $\sum t_i=t$ whose coefficient is nonzero in $\mathbb{F}$, then $f$ is nonzero at some point of $\prod S_i$, where each $S_i$ is a set of $t_i+1$ distinct values in $\mathbb{F}$.
\end{thm}

The first item of Theorem~\ref{thm:mainlist} follows from the following strengthened theorem.

\begin{thm}
Let $G$ be a planar $\{1,3\}$-regular graph of girth at least $41$, and let $p \in V(G)$.
Assign distinct colors to the edges incident to $p$ and let $L$ be a $5$-list-assignment to the remaining edges of $G$.
There exists a strong edge-coloring $c$ where $c(e) \in L(e)$ for all $e \in E(G)$.
\end{thm}

\begin{proof}
For the sake of contradiction, select $G$, $p$, $c$, and $L$ as in the theorem statement, and assume there does not exist a strong edge coloring of $E(G)$ using colors from $L$.
In this selection, minimize $n(G)$.
Note that $G$ is connected and $e(G) > 5$.
We can further assume that $d(p) > 1$, since if $d(p)=1$ and $\{p'\}=N(p)$ then we can instead color the edges incident to $p'$.

\begin{lem}\label{lma:cutedge}
There does not exist a cut-edge $uv$ such that $d(u) = d(v) = 3$.
\end{lem}

%
%
%

\begin{proof}
Suppose that $G$ contains a cut-edge $uv$ with $d(u) = d(v) = 3$.
There are exactly two components in $G - uv$, call them $G_1$ and $G_2$, with $u \in V(G_1)$ and $v \in V(G_2)$.
Without loss of generality, $p \in V(G_1)$.
For each $i \in \{1,2\}$, let $G_i' = G_i + uv$.

Since $d(v) = 3$, $n(G_1') < n(G)$. 
Thus there is a strong edge-coloring of $G_1'$ using the 5-list-assignment $L$.
Next, color the other two edges incident to $v$ using colors distinct from those on the edges incident to $u$.
Now, $G_2'$ is a subcubic planar graph of girth at least 41 with distinctly colored edges about the vertex $v$ and $n(G_2') < n(G)$.
Thus, there is an extension of the coloring to $G_2'$.

The colorings of $G_1'$ and $G_2'$ form a strong edge coloring of $G$, a contradiction.
\end{proof}

Define a \emph{$k$-caterpillar} to be a $k$-thread $v_1,\dots,v_k$ in $G$ where $p \notin \{v_1,\dots,v_k\}$.
Figure~\ref{8cat} is an $8$-caterpillar.

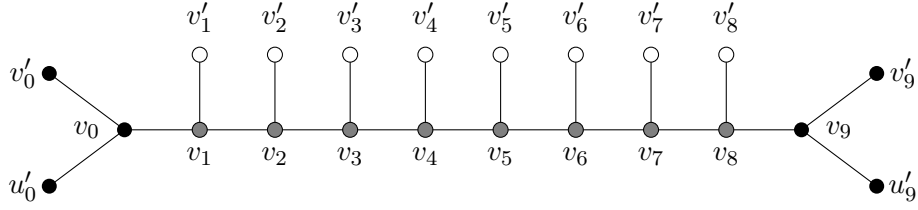
\begin{figure}[h]

	\centering
	\begin{tikzpicture}[scale=1]

\draw (0,0)--(9,0)
(0,0) node[insep]{}
(9,0) node[insep]{};

\foreach \x in {1,2,...,8}{
\draw (\x,0)--(\x,1)
(\x,0) node[free]{}
(\x,1) node[extra]{};
\node at (\x,-0.35) {$v_{\x}$};
\node at (\x,1.5) {$v_{\x}'$};
};

\draw (9,0)--(10,.75)
(9,0)--(10,-.75)
(0,0)--(-1,.75)
(0,0)--(-1,-.75);

\draw (-1,-.75) node[insep]{};
\draw (-1,.75) node[insep]{};
\draw (10,-.75) node[insep]{};
\draw (10,.75) node[insep]{};
\node at (-0.5,0) {$v_{0}$};
\node at (9.5,0) {$v_{9}$};
\node at (10.35,.75) {$v_9'$};
\node at (10.35,-.75) {$u_9'$};
\node at (-1.35,.75) {$v_0'$};
\node at (-1.35,-.75) {$u_0'$};
\end{tikzpicture}

	\caption{An 8-caterpillar.}
	\label{8cat}
\end{figure}

\begin{lem}\label{lma:caterpillar}
$G$ does not contain an $8$-caterpillar.
\end{lem}


\begin{proof} 
We will show that if $G-p$ contains an 8-caterpillar, then $G$ has a strong edge $L$-coloring.
If $v_1,\dots,v_8$ form an 8-caterpillar, then let $v_i'$ be the 1-vertex adjacent to $v_i$, $v_0$ and $v_9$ be the other neighbors of $v_1$ and $v_8$. 
For $i \in \{0,9\}$, let $v_i'$ and $u_i'$ be the neighbors of $v_i$ other than $v_1$ or $v_8$.

By removing all edges incident to $v_2,\dots, v_7$ and $u_1,\dots,u_8$, as well as any isolated vertices that are produced, we obtain a graph $G'$ with fewer vertices than $G$, so we can strongly edge-color $G'$ with 5 colors. 
We fix such a coloring of $G'$ and generate a contradiction by extending this coloring to a strong edge-coloring of $G$. 
Suppose that $c_1, \dots, c_6$ are the colors of the edges incident to the vertices $v_0$ and $v_9$, and assign variables $y_1, \dots, y_8$ to the pendant edges, and variables $x_1, \dots, x_7$ to the interior edges as shown in Figure \ref{fig:lcat}.

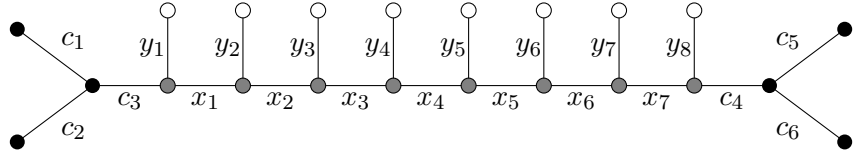
\begin{figure}[h]
	\centering
		\begin{tikzpicture}[scale=1]

\draw (0,0)--(9,0)
(0,0) node[insep]{}
(9,0) node[insep]{};

\foreach \x in {1,2,...,8}{
\draw (\x,0)--(\x,1)
(\x,0) node[free]{}
(\x,1) node[extra]{};
};

\draw (9,0)--(10,.75)
(9,0)--(10,-.75)
(0,0)--(-1,.75)
(0,0)--(-1,-.75);

\foreach \x in {1,2,...,8}{
\draw (\x-.2,.5) node{$y_{\x}$};
}

\foreach \x in {1,2,...,7}{
\draw (\x+.5, -.2) node{$x_{\x}$}; 
}

\draw
(.5,-.2) node{$c_{3}$}
(8.5,-.2) node{$c_{4}$}
(-.25,.6) node{$c_1$}
(-.25,-.6) node{$c_2$}
(9.25,.6) node{$c_5$}
(9.25,-.6) node{$c_6$}
;

\draw (-1,-.75) node[insep]{};
\draw (-1,.75) node[insep]{};
\draw (10,-.75) node[insep]{};
\draw (10,.75) node[insep]{};

\end{tikzpicture}
	\caption{The assignment of colors and variables to the 8-caterpillar.}
\label{fig:lcat}
\end{figure}

Identifying the conflicts between variables and colors produces the following polynomial,
\begin{align*}
f(y_1, \dots, y_8,x_1,\dots, x_7) &= (y_2 - c_3)(x_2 - c_3)(y_7 - c_4)(x_6 - c_4) \\
&\quad\cdot 
\prod_{i=1}^3(x_1 - c_i)\prod_{i=1}^3(y_1 - c_i) \prod_{i=4}^6 (x_7 - c_i) \prod_{i=4}^6 (y_8 - c_i)\\ 
&\quad\cdot \prod_{j -i  \in \{1,2\}}(x_i-x_j)  \prod_{j-i = 1}(y_i - y_j)\prod_{i - j \in \{-1,0,1,2\}} (y_i - x_j). 
\end{align*}

We will use the Combinatorial Nullstellensatz to show that there is an assignment of colors $\hat{c}_1,\dots, \hat{c}_8$ and $c_1',\dots, c_7'$ such that $f(\hat{c}_1,\dots,\hat{c}_8,c_1',\dots,c_7')\ne 0$.  Such an assignment of colors would extend the inductive coloring of $G-p$ to a strong edge-coloring of $G$. If the coefficient of \[(x_{1}~ x_{2}~ x_{3}~ x_{4} ~x_{5} ~x_{6} ~x_{7} ~y_{1} ~y_{2} ~y_{3} ~y_{4} ~y_{5} ~y_{6} ~y_{7} ~y_{8})^4\] is nonzero, then there are values from $L$ for $x_1,\ldots,x_7,y_1,\ldots,y_8$ such that $f$ is nonzero by Theorem~\ref{CN}. 
Using the Magma algebra system~\cite{magma}, this monomial has coefficient $-2$, and thus there is a strong edge-coloring using the 5-list assignment\footnote{All source code and data is available at \url{http://www.math.iastate.edu/dstolee/r/scindex.htm}.}.  
Thus, the 8-caterpillar does not exist in a vertex minimal counterexample. 
\end{proof}

Note that the proof in Lemma~\ref{lma:caterpillar} cannot be extended to exclude a 7-caterpillar in $G$, as there exists a 5-coloring of the external edges that does not extend to the caterpillar, even when the lists are all the same.

To complete the proof, we apply a discharging argument to $\ct(G)$.  \footnote{Our discharging approach is similar to the proof of Lemma~\ref{lma:nrs} where $\ell = 8$, but some care is needed due to the precolored vertex $p$.}. 
First, observe that by Lemma~\ref{lma:cutedge}, $\ct(G)$ is 2-connected and so every face is a simple cycle of length at least 41.
Also observe that by Lemma~\ref{lma:caterpillar}, $\ct(G)$ does not contain a path of length 8 where every vertex is of degree 2, unless one of those vertices is $p$.

Assign charge $2d(v)-6$ to every vertex $v \neq p$, charge $\ell(f) - 6$ to every face $f$, and charge $2d(p)+5$ to $p$. 
By Proposition~\ref{prop:sum}, the total amount of charge on $\ct(G)$ is $-1$. 
Apply the following discharging rules.

\begin{enumerate}[(R1)]
    \item For every $v\in G-p$, if $v$ is a $2$--vertex, $v$ pulls charge 1 from each incident face.
    
    \item If $p$ is a $2$--vertex, then $p$ gives charge $\frac{9}{2}$ to each incident face.
\end{enumerate}

Observe that every vertex has nonnegative charge after this discharging process.
It remains to show that every face has nonnegative charge.

Let $f$ be a face, and let $r_2$ be the number of 2--vertices on the boundary of $f$, not counting $p$, and consider two cases.

\textbf{Case 1}: $d(p)=3$ or $p$ is not adjacent to $f$.

In this case, $p$ does not give charge to $f$, and therefore $f$ has charge $\ell(f) - r_2 - 6$ after discharging. 
Also, the boundary of $f$ does not contain a path of length 8 containing only vertices of degree 2, thus $r_2 \leq \left\lfloor \frac{7}{8}\ell(f)\right\rfloor$. 
Since $\ell(f) \geq 41$, we have
\[
	\ell(f) - r_2 - 6 \geq \ell(f) - \left\lfloor \frac{7}{8}\ell(f)\right\rfloor - 6 \geq 0.
\]

\textbf{Case 2}: $d(p)=2$ and $p$ is adjacent to $f$.

By (R2), $p$ gives charge $\frac{9}{2}$ to $f$, so that $f$ has charge $\ell(f)  - r_2 - \frac{3}{2}$ after discharging.
The boundary of $f$ does not contain a path of length 8 containing only vertices of degree 2, except when using $p$, so, $r_2 \leq \left\lfloor \frac{7}{8}\ell(f)\right\rfloor$.
Since $\ell(f) \geq 41$, we have
\[
	\ell(f) - r_2 - \frac{3}{2} \geq \ell(f) - \left\lfloor \frac{7}{8}\ell(f)\right\rfloor - \frac{3}{2} \geq 0.
\]
Thus, all vertices and faces have nonnegative charge, contradicting Proposition~\ref{prop:sum}.
\end{proof}

The second item of Theorem~\ref{thm:mainlist} follows by similarly strengthening the statement to include a precolored vertex and using Lemmas~\ref{lma:cw}, \ref{lma:cutedge}, and \ref{lma:caterpillar}.


\section{Strong Edge-Coloring of Sparse Graphs}\label{sec:sparse}

In this section, we prove Theorems~\ref{thm:sparse} and \ref{thm:sparse4}.

Let $G$ be a graph with maximum degree $\Delta(G) \leq d$. 
For a vertex $v$ in $\ct(G)$ denote by $N_3(v)$ the set of $3^+$-vertices $u$ where $\ct(G)$ contains a path $P$ from $u$ to $v$ where all internal vertices of $P$ are $2\bla$-vertices.  
For $u \in N_3(v)$, let $\mu(v,u)$ be the number of paths from $v$ to $u$ whose internal vertices have degree 2 in $\ct(G)$. 
For a 3-vertex $v$, let the \emph{responsibility set}, denoted $\Resp(v)$, be the set of $2\bla$-vertices that appear on the paths between $v$ and the vertices in $N_3(v)$.

Let $D$ be a subgraph of $G$. We call $D$ a \emph{$k$-reducible configuration} if there exists a subgraph $D'$ of $D$ such that
any  strong $k$-edge-coloring of $G-D'$ can be extended to a strong $k$-edge-coloring of $G$.
One necessary property for the selection of $D'$ is that no two edges that remain in $G-D'$ can have distance at most two in $G$ but distance strictly larger than two in $G - D'$.
In the next subsection we describe several reducible configurations. 

\subsection{Reducible Configurations}

This subsection contains description of four types of reducible configurations. 
Each configuration is described in terms of how it appears within $\ct(G)$ where $G$ is a graph with maximum degree $\Delta(G) \leq d$ for some $d \geq 4$.

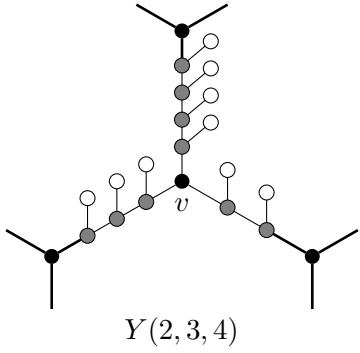
\begin{figure}[p]
\centering
{\usetikzlibrary[topaths]

\tikzset{
insep/.style={inner sep=2pt, outer sep=0pt, circle,fill}, 
free/.style={inner sep=2pt, outer sep=0pt, circle,fill=gray,draw}, 
extra/.style={inner sep=2pt, outer sep=0pt, circle,fill=white,draw}, 
}


\begin{tikzpicture}[scale=1]
\draw
(0,0) node[insep,label=below:{$v$}](a){} 
(a)++(90:2)  node[insep](b){} 
(a)++(210:2)  node[insep](c){} 
(a)++(90-120:2)  node[insep](d){} 
(a) to node[free,pos=0.2](x1){} node[free,pos=0.40](x2){} node[free,pos=0.6](x3){} node[free,pos=0.8](x4){}  (b)
(a) to  node[free,pos=0.25](x5){} node[free,pos=0.50](x6){} node[free,pos=0.75](x7){}   (c)
(a) to  node[free,pos=0.333](x8){} node[free,pos=0.666](x9){}    (d)
;
\draw[line width=1pt]
(b) -- (x4)
(b) -- ++(150:0.7)
(b) -- ++(30:0.7)
(c) -- (x7)
(c) -- ++(270:0.7)
(c) -- ++(150:0.7)
(d) -- (x9)
(d) -- ++(30:0.7)
(d) -- ++(270:0.7)
;
\draw
\foreach \x in {5,6,...,9} { (x\x) -- ++(0,0.5) node[extra]{} }
\foreach \x in {1,2,...,4} { (x\x) -- ++(40:0.5) node[extra]{} }
;
\draw(0,0) ++(0,-2) node{$Y(2,3,4)$};
\end{tikzpicture}


}
\caption{The configuration $Y(t_1,t_2,t_3)$.}\label{fig:Y}
\end{figure}

\begin{figure}[p]
\centering
{\usetikzlibrary[topaths]

\tikzset{
insep/.style={inner sep=2pt, outer sep=0pt, circle,fill}, 
free/.style={inner sep=2pt, outer sep=0pt, circle,fill=gray,draw}, 
extra/.style={inner sep=2pt, outer sep=0pt, circle,fill=white,draw}, 
}

\begin{tikzpicture}[scale=1]
\draw
(-2,0) node[insep,label=left:{$v$}](a){} 
(2,0) node[insep,label=right:{$u$}](b){}
(a) to node[free,pos=0.2](x1){} node[free,pos=0.40](x2){} node[free,pos=0.6](x3){} node[free,pos=0.8](x4){}  (b)
\foreach \x in {1,2,...,4}{(x\x) -- ++(0,0.5) node[extra]{}}

(a)++(0,2)  node[insep](c){} 
(a)++(0,-2)  node[insep](d){} 
(a) to  node[free,pos=0.333](x4){} node[free,pos=0.666](x1){}    (c)
(a) to  node[free,pos=0.333](x3){} node[free,pos=0.666](x2){}    (d)
\foreach \x in {1,2,...,4}{(x\x) -- ++(-0.5,0) node[extra]{}}
;
\draw[line width=1pt]
(c) -- (x1)
(c) -- ++(150:0.7)
(c) -- ++(30:0.7)
(d) -- (x2)
(d) -- ++(210:0.7)
(d) -- ++(330:0.7)
;

\draw
(b)++(0,2)  node[insep](c){} 
(b)++(0,-2)  node[insep](d){} 
(b) to  node[free,pos=0.25](x5){} node[free,pos=0.50](x3){} node[free,pos=0.75](x1){}   (c)
(b) to  node[free,pos=0.25](x6){} node[free,pos=0.50](x4){} node[free,pos=0.75](x2){}   (d)
\foreach \x in {1,2,...,6}{(x\x) -- ++(0.5,0) node[extra]{}}
;
\draw[line width=1pt]
(c) -- (x1)
(c) -- ++(150:0.7)
(c) -- ++(30:0.7)
(d) -- (x2)
(d) -- ++(210:0.7)
(d) -- ++(330:0.7)
;

\draw(0,0) ++(0,-2) node{$H(2,2;4;3,3)$};
\end{tikzpicture}


}
\caption{The configuration \label{fig:H}$H(t_1,t_2;r;s_1,s_2)$.}
\end{figure}
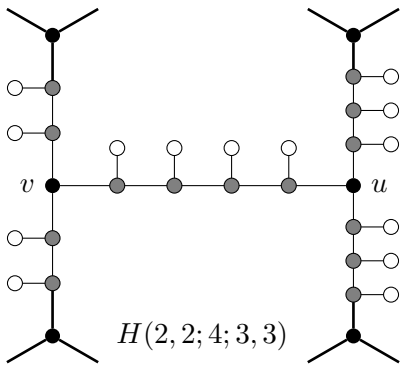

\begin{figure}[p]
\centering
{\usetikzlibrary[topaths]

\tikzset{
insep/.style={inner sep=2pt, outer sep=0pt, circle,fill}, 
free/.style={inner sep=2pt, outer sep=0pt, circle,fill=gray,draw}, 
extra/.style={inner sep=2pt, outer sep=0pt, circle,fill=white,draw}, 
}

\begin{tikzpicture}[scale=1]
\draw
(-2,0) node[insep,label=below:{$v$}](a){} 
(2,0) node[insep,label=below:{$u$}](b){}
(a)++(-3,0)  node[insep](c){} 
(b)++(3,0)  node[insep](d){} 
(a) to[bend left=50] node[free,pos=0.2](x1){} node[free,pos=0.40](x2){} node[free,pos=0.6](x3){} node[free,pos=0.8](x4){}  (b)
(a) to[bend right=50] node[free,pos=0.2](x5){} node[free,pos=0.4](x6){} node[free,pos=0.6](x7){}  node[free,pos=0.8](x8){}  (b)
(a) to  node[free,pos=0.25](x9){} node[free,pos=0.50](x10){} node[free,pos=0.75](x11){}   (c)
(b) to  node[free,pos=0.333](x12){} node[free,pos=0.666](x13){}    (d)
;
\draw[line width=1pt]
(c) -- (x11)
(c) -- ++(120:0.7)
(c) -- ++(240:0.7)
(d) -- (x13)
(d) -- ++(60:0.7)
(d) -- ++(-60:0.7)
;
\draw
\foreach \x in {1,2,...,13}
{
(x\x) -- ++(0,0.5) node[extra]{}
}
;
\draw(0,0) ++(0,-2) node{$\Phi(3,4,4,2)$};
\end{tikzpicture}

}
\caption{The configuration \label{fig:Phi}$\Phi(t,a_1,a_2,s)$.}
\end{figure}
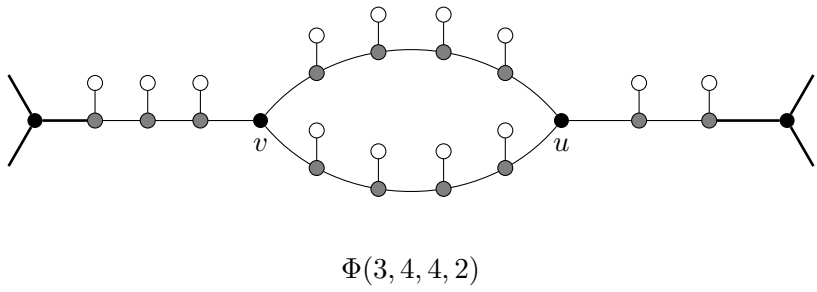

Let $t$ be a positive integer.
The \emph{$t$-caterpillar} is formed by two $3^+$-vertices $v_0$ and $v_{t+1}$ with a path $v_0v_1\dots v_tv_{t+1}$ where each $v_i$ is a $2\bla$-vertex for every $i \in \{1,\dots,t\}$.

Let $t_1,\dots, t_k$ be nonnegative integers.
A configuration $Y(t_1,\dots,t_k)$ is formed by a $k^+$-vertex $v$ and $k$ internally disjoint paths of lengths $t_1+1,\dots,t_k+1$ with $v$ as a common endpoint, where the internal vertices of the paths are $2\bla$-vertices.
We call such configuration a \emph{$Y$-type configuration about $v$}, see Figure~\ref{fig:Y}.

A configuration $H(t_1,t_2;r;s_1,s_2)$ is formed by two 3-vertices $u$ and $v$
and 5 internally disjoint paths of lengths   $t_1+1$, $t_2+1$, $r+1$, $s_1+1$, and $s_2+1$, 
where the internal vertices of the paths are $2\bla$-vertices. The paths of lengths
$t_1+1$ and $t_2+1$ have  $v$ as an endpoint, the path of length $r+1$ has  $u$ and $v$ as endpoints 
and the paths of lengths $s_1+1$ and $s_2+1$ have $u$ as an endpoint.
We call such configuration an \emph{$H$-type configuration about $v$ and $u$}, see Figure~\ref{fig:H}.

A configuration $\Phi(t,a_1,a_2,s)$ is formed by two 3-vertices $u$ and $v$
and 4 internally disjoint paths of lengths   $t+1$, $a_1+1$, $a_2+1$, and $s+1$, 
where the internal vertices of the paths are $2\bla$-vertices.
The path of length $t+1$ has $v$ as an endpoint, the paths of lengths $a_1+1$ and $a_2+1$ have $u$ and $v$ as endpoints
and the path of length $s+1$ has $u$ as an endpoint.
We call such configuration a \emph{$\Phi$-type configuration about $v$ and $u$}, see Figure~\ref{fig:Phi}.

The reducibility of these configurations was verified using computer\footnote{All source code and data is available at \url{http://www.math.iastate.edu/dstolee/r/scindex.htm}.}, and in addition the 8-caterpillar is addressed in Lemma~\ref{lma:caterpillar}.  
Given the definition of a $2\bla$-vertex, the vertices of degree two in these configurations may, or may not, be adjacent to some 1-vertices in $G$.  
We demonstrate the reducibility of the instances of these configurations wherein each vertex of degree 2 is adjacent to $d-2$ 1-vertices, as depicted in Figures \ref{fig:Y}--\ref{fig:Phi}.  
This suffices to address all other instances of these configurations that may occur.  

\begin{clm}\label{claim:red_cat}
The following caterpillars with maximum degree $d$ are reducible:
\begin{enumerate}
\item (Borodin and Ivanova~\cite{BI}) For $d = 3$, the $8$-caterpillar is 5-reducible.
\item (Wang and Zhao~\cite{WZ}) For $d \geq 4$, the  $(2d-2)$-caterpillar is $(2d-1)$-reducible.
\end{enumerate}
\end{clm}

These caterpillars are likely the smallest that are reducible for each degree $d$.
Thus, the bounds in Theorems~\ref{thm:bi} and \ref{thm:wz} are best possible using only Lemma~\ref{lma:nrs}.
To improve these bounds, we demonstrate larger reducible configurations and use a more complicated discharging argument.

\begin{clm}\label{clm:Reducible}
The following configurations with maximum degree 3 are 5-reducible:
\begin{enumerate}
\item $Y(1,6,7)$, $Y(2,5,6)$ and $Y(3,4,5)$.
\item $H(7,7;0;3,7),\, H(7,7;0;4,6),\, H(7,7;0;5,5),\, H(6,7;0;3,7),\, H(6,7;0;4,6),\\
        H(6,7;0;5,5),\, H(6,6;1;2,7),\, H(6,6;1;3,6),\, H(6,6;1;4,5),\, H(5,7;1;2,7),\\
        H(5,7;1;3,6),\, H(5,7;1;4,5),\, H(4,7;2;1,7),\, H(4,7;2;2,6),\, H(4,7;2;3,5),\\
        H(4,7;2;4,4),\, H(3,7;3;1,6),\, H(3,7;3;2,5) \text{ and }  H(3,7;3;3,4).$
\item   $\Phi(7,0,7,1),\, \Phi(7,0,6,1),\, \Phi(6,0,7,1),\, \Phi(6,1,6,1),\, \Phi(7,1,5,1), \Phi(5,1,7,1),\\
         \Phi(7,2,4,1),\, \Phi(4,2,7,1),\, \Phi(7,3,3,1),\, \Phi(3,3,7,1) \text{ and }\Phi(3,7,0,7).$
\end{enumerate}
\end{clm}

\begin{clm}\label{clm:Reducible4}
The following configurations with maximum degree 4 are $7$-reducible:
\[ Y(2,4,4),\ Y(1,5,5),\ Y(2,4,5),\ Y(3,4,4),\text{ and }Y(2,5,5).
\]
\end{clm}
    
 \subsection{Proof of Theorem~\ref{thm:sparse}}

 \begin{proof}
Among graphs $G$ with $\mad(G) < 2 + \frac{1}{7}$ not containing $S_3$, $S_4$, or $S_7$, with $\sci > 5$, select $G$ while minimizing the number of vertices in $\ct(G)$.
Note that  $e(G) > 5$ since $\sci > 5$, and let $n$ be the number of vertices in $\ct(G)$. 
By using the discharging method, we will show that $\mad(\ct(G)) \geq 2+\frac{1}{7}$, which is a contradiction, so no such minimal counterexample exists.
        
Observe that $G$ does not contain any of the reducible configurations addressed in Claim \ref{clm:Reducible}.
We also have the following additional structure on $\ct(G)$.
      
\begin{lem}\label{lma:cutedge2}
$\ct(G)$ is 2-connected.
\end{lem}

\begin{proof}
Suppose that $\ct(G)$ contains a cut-edge $uv$.
In $G$, the vertices $u$ and $v$ have degree at least two.
There are exactly two components, $G_1$ and $G_2$, in $G - uv$, with $u \in V(G_1)$ and $v \in V(G_2)$.
Let $u_1,u_2$ be neighbors of $u$ in $G_1$ and $v_1,v_2$ be neighbors of $v$ in $G_2$; let $u_1 = u_2$ only when $u$ has a unique neighbor in $G_1$, and $v_1 = v_2$ only when $v$ has a unique neighbor in $G_2$.
Let $G_1' = G_1 + \{ uv, vv_1, vv_2\}$ and $G_2' = G_2 + \{ uv, uu_1, uu_2\}$.

If $G_1' = G$, then consider $G' = G - v_1 - v_2$.
Since $n(G') < n(G)$ and $\mad(G') \leq \mad(G)$, there is a strong 5-edge-coloring $c$ of $G'$.
Extend the coloring $c$ to color $c(vv_1)$ and $c(vv_2)$ from the colors not in $\{ c(uv), c(uu_1), c(uu_2)\}$, a contradiction.
We similarly reach a contradiction when $G_2' = G$.

Therefore, $n(G_i') < n(G)$ and $\mad(G_i') \leq \mad(G)$ for each $i \in \{1,2\}$.
Thus, there exist strong 5-edge-colorings $c_1$ and $c_2$ of $G_1'$ and $G_2'$, respectively.
For each coloring, the colors on the edges $uv, uu_1, uu_2, vv_1, vv_2$ are distinct.
Let $\pi$ be a permutation of the five colors satisfying $\pi(c_2(e)) = c_1(e)$ for each edge $e \in \{uv, uu_1, uu_2, vv_1, vv_2\}$.
Then, we extend the coloring $c_1$ of $G_1'$ to all of $G$ by assigning $c_1(e) = \pi(c_2(e))$ for all edges $e \in E(G_2')$.
The coloring $c_1$ is a strong 5-edge-coloring of $G$, a contradiction.
\end{proof}
       
        If $\ct(G)$ does not have any $3$-vertices, then $\ct(G)$ must be isomorphic to cycle $C_n$. 
        If $n \geq 9$, then $\ex_3(G)$ contains an 8-caterpillar.
        If $n \in \{5,6,8\}$, then $G$ is a subgraph of $S_5$, $S_6$, or $S_8$, which each has a strong edge-coloring using five colors, discovered by computer. 
        When $n \in\{3,4,7\}$, $G$ does not contain $S_3$, $S_4$, or $S_7$, and any proper subgraph of these graphs is $5$ strong edge-colorable, discovered by computer.
        Therefore, $\ct(G)$ is not isomorphic to a cycle, and hence for every $2\bla$-vertex $u$ in $G$, $|N_3(u)| \geq 1$.
        
        If $G$ has some vertex $v$ such that $|N_3(v)|=1$, then $G$ must be a subgraph of $\Theta(t_1,t_2,t_3)$, which is the graph consisting of three internally disjoint $x-y$ paths of length $t_1+1, t_2+1$ and $t_3+1$, for some $0 \leq t_1 \leq t_2 \leq t_3$.
 
        If $t_3 \geq 8$, then $\ex_3(G)$ contains an 8-caterpillar, so we assume that $t_3 < 8$. 
        Observe that if $\mad(\Theta(t_1,t_2,t_3)) < 2 + \frac{1}{7}$, then $t_1+t_2+t_3 \geq 13$.
        However, if $\Theta(t_1,t_2,t_3)$ does not contain a reducible $Y$-type configuration, then by Claim~\ref{clm:Reducible} the sequence $(t_1,t_2,t_3)$ is one of $(0,7,7)$, $(0,6,7)$, $(1,6,6)$, $(1,5,7)$, $(2,5,6)$, $(2,4,7)$, or $(3,3,4)$.
        In each of these cases, we have verified by computer that $\Theta(t_1,t_2,t_3)$ has a strong edge-coloring using five colors.
        
        Therefore, $|N_3(v)|\geq 2$ for every $v\in\ct(G)$. 
        We proceed using discharging.
        Assign each vertex initial charge $d(v)$.
        Note that the total charge on the graph is $2e(\ct(G))$, which is at most $\mad(G)n < (2+\frac{1}{7})n$.
        We shall distribute charge among the vertices of $\ct(G)$ and result with charge at least $2 + \frac{1}{7}$ on every vertex, giving a contradiction.
        
        Distribute charge among the vertices according to the following discharging rules, applied to each pair of vertices $u, v \in V(\ct(G))$:
        \begin{enumerate}[(R1)]
            \item\label{2vtx} If $u$ is a 2-vertex and $v \in N_3(u)$, then $v$ sends charge $\frac{1}{14}$ to $u$.
            \item\label{3vtx} If $v$ is a 3-vertex with $|\Resp(v)|\leq 10$ and $u \in N_3(v)$, then
                \begin{enumerate}[(a)]
                    \item\label{adj} if $d(u,v)=1$ and $|\Resp(u)| = 14$,  then $v$ sends charge $\frac{1}{7}$ to $u$;
                    \item\label{near} otherwise, if $d(u,v) \leq 4$, then $v$ sends charge $\frac{1}{14}$ to $u$.
                \end{enumerate}
        \end{enumerate}
        
        We will now verify the assertion that each vertex has final charge at least $2 + \frac{1}{7}$.
	If $v$ is a 2-vertex, then since $|N_3(v)| = 2$ the final charge on $v$ is $2 + \frac{1}{7}$ after by Rule~R\ref{2vtx}.	 Let $v$ be a 3-vertex.	If $u \in N_3(v)$, then $d(u,v) \leq 8$ by Lemma~\ref{lma:caterpillar}.
        Claim~\ref{clm:Reducible} implies that $|\Resp(v)|\leq 14$. 
        
        \begin{mycases}
        
        \mycase{$|\Resp(v)| \in \{11, 12\}$.} In this case, $v$ only loses charge by Rule~R\ref{2vtx}, so the final charge is at least $3 - \frac{12}{14} = 2 + \frac{1}{7}$.
        
        \mycase{$|\Resp(v)| = 14$.} By Claim~\ref{clm:Reducible}, the $Y$-type configuration about $v$ is $Y(0,7,7)$.
        Thus, some vertex $u_1 \in N_3(v)$ is at distance one from $v$.
        If $\mu(v,u_1) = 1$, then the $H$-type configuration about $v$ and $u_1$ is of the form $H(7,7;0;s_1,s_2)$; by Claim~\ref{clm:Reducible} $s_1+s_2 \leq 9$, $|\Resp(u_1)|\leq 9$, and  $u_1$ sends charge $\frac{1}{7}$ to $v$ by Rule~R\ref{adj}.
        If $\mu(v,u_1) = 2$, then the $\Phi$-type configuration about $v$ and $u_1$ is of the form $\Phi(7,0,7,s)$; by Claim~\ref{clm:Reducible} $s = 0$, $|\Resp(u_1)|\leq 7$, and  $u_1$ sends charge $\frac{1}{7}$ to $v$ by Rule~R\ref{adj}.

         \mycase{$|\Resp(v)| = 13$.} By Claim~\ref{clm:Reducible}, the $Y$-type configuration $Y(t_1,t_2,t_3)$ about $v$ is one of $Y(0,6,7)$,  $Y(1,6,6)$, $Y(1,5,7)$, $Y(2,4,7)$, or  $Y(3,3,7)$.
         We consider each case separately.
                  
         \begin{subcases}
         	\subcase{$(t_1,t_2,t_3) = (0,6,7).$} Let $u_1$ be the vertex in $N_3(v)$ at distance 1 from $v$.
        If $\mu(v,u_1) = 1$, then the $H$-type configuration about $v$ and $u_1$ is of the form $H(6,7;0;s_1,s_2)$; by Claim~\ref{clm:Reducible} $s_1+s_2 \leq 9$, $|\Resp(u_1)|\leq 9$, and $u_1$ sends charge $\frac{1}{14}$ to $v$ by Rule~R\ref{near}.
        If $\mu(v,u_1) = 2$, then the $\Phi$-type configuration about $v$ and $u_1$ is of the form $\Phi(6,0,7,s)$ or $\Phi(7,0,6,s)$; by Claim~\ref{clm:Reducible} $s = 0$, $|\Resp(u_1)|\leq 7$, and  $u_1$ sends charge $\frac{1}{14}$ to $v$ by Rule~R\ref{near}.
        
        \subcase{$(t_1,t_2,t_3) = (1,6,6).$} Let $u_1$ be the vertex in $N_3(v)$ at distance 2 from $v$.
        If $\mu(v,u_1) = 1$, then the $H$-type configuration about $v$ and $u_1$ is of the form $H(6,6;1;s_1,s_2)$; by Claim~\ref{clm:Reducible} $s_1+s_2 \leq 8$, $|\Resp(u_1)|\leq 9$, and  $u_1$ sends charge $\frac{1}{14}$ to $v$ by Rule~R\ref{near}.
        If $\mu(v,u_1) = 2$, then the $\Phi$-type configuration about $v$ and $u_1$ is of the form $\Phi(6,1,7,s)$ or $\Phi(7,1,6,s)$; by Claim~\ref{clm:Reducible} $s = 0$, $|\Resp(u_1)|\leq 8$, and  $u_1$ sends charge $\frac{1}{14}$ to $v$ by Rule~R\ref{near}.
        
        \subcase{$(t_1,t_2,t_3) = (1,5,7).$} Let $u_1$ be the vertex in $N_3(v)$ at distance 2 from $v$.
        If $\mu(v,u_1) = 1$, then the $H$-type configuration about $v$ and $u_1$ is of the form $H(5,7;1;s_1,s_2)$; by Claim~\ref{clm:Reducible} $s_1+s_2 \leq 8$, $|\Resp(u_1)|\leq 9$, and  $u_1$ sends charge $\frac{1}{14}$ to $v$ by Rule~R\ref{near}.
        If $\mu(v,u_1) = 2$, then the $\Phi$-type configuration about $v$ and $u_1$ is of the form $\Phi(5,1,7,s)$ or $\Phi(7,1,5,s)$; by Claim~\ref{clm:Reducible} $s = 0$, $|\Resp(u_1)|\leq 8$, and  $u_1$ sends charge $\frac{1}{14}$ to $v$ by Rule~R\ref{near}.

        \subcase{$(t_1,t_2,t_3) = (2,4,7).$} Let $u_1$ be the vertex in $N_3(v)$ at distance 3 from $v$.
        If $\mu(v,u_1) = 1$, then the $H$-type configuration about $v$ and $u_1$ is of the form $H(4,7;2;s_1,s_2)$; by Claim~\ref{clm:Reducible} $s_1+s_2 \leq 7$, $|\Resp(u_1)|\leq 9$, and $u_1$ sends charge $\frac{1}{14}$ to $v$ by Rule~R\ref{near}.
        If $\mu(v,u_1) = 2$, then the $\Phi$-type configuration about $v$ and $u_1$ is of the form $\Phi(4,2,7,s)$ or $\Phi(7,2,4,s)$; by Claim~\ref{clm:Reducible} $s = 0$, $|\Resp(u_1)|\leq 8$, and  $u_1$ sends charge $\frac{1}{14}$ to $v$ by Rule~R\ref{near}.

        \subcase{$(t_1,t_2,t_3) = (3,3,7).$} Let $u_1$ be the vertex in $N_3(v)$ at distance 4 from $v$.
        If $\mu(v,u_1) = 1$, then the $H$-type configuration about $v$ and $u_1$ is of the form $H(3,7;3;s_1,s_2)$; by Claim~\ref{clm:Reducible} $s_1+s_2 \leq 7$, $|\Resp(u_1)|\leq 10$, and  $u_1$ sends charge $\frac{1}{14}$ to $v$ by Rule~R\ref{near}.
        If $\mu(v,u_1) = 2$, then the $\Phi$-type configuration about $v$ and $u_1$ is of the form $\Phi(3,3,7,s)$ or $\Phi(7,3,3,s)$; by Claim~\ref{clm:Reducible} $s = 0$, $|\Resp(u_1)|\leq 10$, and  $u_1$ sends charge $\frac{1}{14}$ to $v$ by Rule~R\ref{near}.
        
        \end{subcases}
        
        \mycase{$|\Resp(v)| \leq 10$.}
        In this case, $v$ loses charge at most $\frac{10}{14}$ by Rule~R\ref{2vtx}, so if it sends charge at most $\frac{1}{7}$ by Rule~R\ref{3vtx}, then the final charge on $v$ is at least $2 + \frac{1}{7}$.
        Consider how much charge is sent by Rule~R\ref{3vtx}.
        
        \begin{subcases}
        \subcase{$v$ sends charge $\frac{3}{14}$ by Rule~R\ref{3vtx}.}
        		If $|\Resp(v)| \leq 9$, then the final charge on $v$ is at least $2 + \frac{1}{7}$, so assume that $|\Resp(v)| = 10$.
        		If $v$ sends charge $\frac{1}{14}$ to each of three vertices in $N_3(v)$, then $d(v,u) \leq 4$ for each $u \in N_3(v)$ and hence $|\Resp(v)| < 10$.		
		Thus, $v$ sends charge $\frac{1}{7}$ to some $u_1 \in N_3(v)$ and $\frac{1}{14}$ to some $u_2 \in N_3(v)$.
		Since $|\Resp(u_1)| = 14$, Claim~\ref{clm:Reducible} implies that the $Y$-type configuration about $u_1$ is of the form $Y(0,7,7)$.
		Since $v$ is adjacent to $u_1$, $d(v,u_2)\leq 4$, and $|\Resp(v)| = 10$, the $Y$-type configuration about $v$ is of the form $Y(0,3,7)$.
		If $\mu(v,u_1)= 1$, then the $H$-type configuration about $v$ and $u_1$ is of the form $H(3,7;0;7;7)$ which is reducible by Claim~\ref{clm:Reducible}.
		If $\mu(v,u_1)= 2$, then the $\Phi$-type configuration about $v$ and $u_1$ is of the form $\Phi(3,7,0,7)$ which is reducible by Claim~\ref{clm:Reducible}.

        \subcase{$v$ sends charge $\frac{2}{7}$ by Rule~R\ref{3vtx}.}
		In this case, $v$ must send charge $\frac{1}{7}$ to at least one vertex $u_1$ in $N_3(v)$.	
		If $v$ sends charge $\frac{1}{7}$ to another vertex $u_2$ in $N_3(v)$, then, as $G$ contains no $8$-caterpillar, $|\Resp(v)| \leq 7$ and hence the final charge on $v$ is at least $2 + \frac{3}{14}$.
		If $v$ sends charge $\frac{1}{14}$ to the other two vertices $u_2$ and $u_3$ in $N_3(v)$, then $|\Resp(v)| \leq 6$ and hence the final charge on $v$ is at least $2 + \frac{5}{14}$.

	\subcase{$v$ either sends charge $\frac{5}{14}$ or $\frac{3}{7}$ by Rule~R\ref{3vtx}.}	
		Suppose that $v$ sends charge $\frac{5}{14}$ by Rule~R\ref{3vtx}. Thus, $v$ must send charge $\frac{1}{7}$ to two of three vertices in $N_3(v)$, and $\frac{1}{14}$ to the third vertex.  This implies that $|\Resp(v)|\le 3$ and hence the final charge on $v$ is at least $2 + \frac{3}{7}$.  Similarly, if $v$ sends charge $\frac{3}{7}$ by Rule~R\ref{3vtx}, then $|Resp(v)|=0$.  Thus, the final charge on $v$ is $2+\frac{4}{7}$.  
        \end{subcases}

        \end{mycases}
        
        In all cases, we verified that the final charge is at least $2 + \frac{1}{7}$, contradicting that the average degree of $\ct(G)$ is strictly less than $2 + \frac{1}{7}$.
    \end{proof}
    
We note that it is possible to improve the bound $\mad(G) < 2 + \frac{1}{7}$ by a small amount.
In particular, the discharging method used above essentially states that the average size of a responsibility set in $\ct(G)$ is at most 12.
By careful analysis, we can find that a 3-vertex $v$ with $|\Resp(v)|\leq 11$ has some excess charge after the discharging argument that could be used to increase the charge on nearby vertices by a small fraction.
We have verified using computation that for every 3-vertex $v$, there is at least one vertex $u \in N_3(v)$ where $|\Resp(u)| < 12$.
Thus, it is impossible to have a minimal counterexample where all responsibility sets have size 12, and it is feasible to construct a discharging argument that will improve on the bound $\mad(G) < 2 + \frac{1}{7}$ by a small fraction.
We do not do this explicitly as it requires significant detail without significant gain.

In order to prove that $\mad(G) < 2 + \frac{1}{6}$ implies that $G$ can be strongly 5-edge-colored, then the proof will imply that the average size of a responsibility set is at most 10.
This will require sending charge to all of the vertices with 11 or 12 vertices in the responsibility set, and also making sure that the charge comes from vertices with responsibility sets much smaller.
Likely, larger reducible configurations will grant some improvement in this direction, but our algorithm is insufficient to effectively test reducibility for larger configurations.\\

\subsection{Proof of Theorem~\ref{thm:sparse4}}

\begin{proof}
Note that the second item of Theorem~\ref{thm:sparse4} follows from the first by Proposition~\ref{prop:mad}.
For the first item, we follow a similar discharging argument as in Theorem~\ref{thm:sparse}.
The argument will be simpler as we will only discharge from $3^+$-vertices to $2\bla$-vertices.
Select a graph $G$ that satisfies the hypotheses and minimizes $n(G)$.
Observe that $\ct(G)$ is 2-connected by an argument similar to Lemma~\ref{lma:cutedge2}.

Since the $6$-caterpillar is $7$-reducible by Claim \ref{claim:red_cat}, $\ct(G)$ does not contain a path of six $2\bla$-vertices.
Since $G$ has girth at least 7, $\ct(G)$ is not a cycle, so it contains at least one $3^+$-vertex.


If $v$ is a $3^+$-vertex, then let $\Resp(v)$ be the set of $2\bla$-vertices reachable from $v$ using only $2\bla$-vertices.
We consider $\Resp(v)$ to be a multiset, where the multiplicity of a vertex $u \in \Resp(v)$ is given by the number of paths from $v$ to $u$ using only $2\bla$-vertices. 
Note that the multiplicity is either 1 or 2.

Assign charge $d_{\ct(G)}(v)$ to each vertex $v \in V(\ct(G))$.
Note that the average charge on each vertex is equal to the average degree of $G$.
To discharge, let $\varepsilon = \frac{1}{13}$ and each $3^+$-vertex $v$ sends $\varepsilon m$ to each $2\bla$-vertex in $\Resp(v)$ with multiplicity $m$.
Thus, every $2\bla$-vertex ends with charge $2 + \frac{2}{13}$.

Suppose $d_{\ct(G)}(v) = 3$.
Since $\ct(G)$ is 2-connected, all vertices in $\Resp(v)$ appear with multiplicity one.
By Claim~\ref{clm:Reducible4}, $|\Resp(v)| \leq 11$. 
Thus each $3$-vertex ends with charge  at least $3 - \frac{11}{13} = 2 + \frac{2}{13}$. 

Suppose $d_{\ct(G)}(v) = 4$.
Since the $(6,4)$-caterpillar is reducible, each path of $2\bla$-vertices has length at most five, and hence $|\Resp(v)| \leq 20$, including multiplicity. 
Thus each $4$-vertex ends with charge  at least $4 - \frac{20}{13} = 2 + \frac{6}{13} > 2 + \frac{2}{13}$. 

Therefore, every vertex ends with charge at least $2 + \frac{2}{13}$ and thus the average degree of $G$ is at least $2 + \frac{2}{13}$, a contradiction.
\end{proof}
%
%
%
%
%
%


\begin{thebibliography}{99}
{
\small
\frenchspacing
\setlength{\itemsep}{0em}
\setlength{\parskip}{0.25em}

\bibitem{Alon95}
N.~Alon.
\newblock Combinatorial {N}ullstellensatz.
\newblock {\em Combin. Probab. Comput.} \textbf{8} (1999), 7--29.

\bibitem{Andersen}
L. D. Andersen, The strong chromatic index of a cubic graph is at most 10, \textit{Discrete Math.} 108 (1992) 231--252.


\bibitem{BI} O.V. Borodin and A.O. Ivanova, Precise upper bound for the strong edge chromatic number of sparse planar graphs. \textit{Discussiones Mathematicae Graph Theory}, 33(4) (2014) 759--770.


\bibitem{magma} W. Bosma, J. Cannon, and C. Playoust, The Magma algebra system. I. The user language. \textit{J. Symbolic Comput.} \textbf{24} (1997) 235--265.

\bibitem{BJ} H. Bruhn, F. Joos, A stronger bound for the strong chromatic index. \textit{arXiv preprint} \texttt{arXiv:1504.02583}.

\bibitem{CMPR} J. Chang, M. Montassier, A. P\v{e}che, and A. Raspaud, Strong chromatic index of planar graphs with large girth. \textit{Discussiones Mathematicae Graph Theory}, 34(4), (2014) 723--733.

\bibitem{CW} D.W. Cranston and D.B. West, A Guide to the Discharging Method. \textit{arXiv preprint} \texttt{arXiv:1306.4434}.

\bibitem{EN1988}
P.~Erd\H{o}s,
Problems and results in combinatorial analysis and graph theory,
\textit{Proceedings of the First Japan Conference on Graph Theory and Applications} 
(Hakone, 1986), 
\textbf{72} (1988), 81--92.

\bibitem{Faudree90}
R.J. Faudree, A. Gy\'arfas, R.H. Schelp, and Zs. Tuza. 
\newblock The strong chromatic index of graphs, 
\newblock {\em Ars Combin.} \textbf{29} (1990) (B), 205--211.

\bibitem{Fouquet83}
J.-L. Fouquet and J.-L. Jolivet, 
\newblock Strong edge-colorings of graphs and applications to multi-k-gons, 
\newblock {\em Ars Combin.} \textbf{16} (1983) (A) 141--150.

\bibitem{Fouquet84}
J.-L. Fouquet and J.-L. Jolivet, 
\newblock Strong edge-coloring of cubic planar graphs, in Progress in graph theory (Waterloo, Ont., 1982), 
\newblock Academic Press, Toronto, ON, 1984, pp. 247--264.

\bibitem{Hocquard11}
H. Hocquard and P. Valicov, 
\newblock Strong edge colouring of subcubic graphs, 
\newblock {\em Discrete Appl. Math.} \textbf{159} (2011), 1650--1657.

\bibitem{Hocquard13}
H. Hocquard, M. Montassier, A. Raspaud, and P. Valicov,
\newblock On strong edge-colouring of subcubic graphs
\newblock {\em Discrete Appl. Mathematics} \textbf{161} (2013), 2467--2479.

\bibitem{Hudak}
D. Hud\'ak, B. Lu\v{z}ar, R. Sot\'ak, and R. \v{S}krekovski, Strong edge-coloring of planar graphs, \textit{Discrete Math.} \textbf{324} (2014), 41--49.

\bibitem{KLRSWY2014} A.V. Kostochka, X. Li, W. Ruksasakchai, M. Santana, T. Wang, and G. Yu, Strong chromatic index of subcubic planar multigraphs, \textit{in preparation}.

\bibitem{MR} M. Molloy and B. Reed. A bound on the strong chromatic index of a graph, \textit{J. Combin. Theory, Ser. B} \textbf{69} (1997), 103--109.

\bibitem{NRS} J. Ne\v{s}et\v{r}il, A. Raspaud, and E. Sopena, Colorings and girth of oriented planar graphs, \textit{Discrete Math.} 165/166 (1997) 519--530.

\bibitem{StegerYu}
A. Steger and M.-L. Yu, On induced matchings, \textit{Discrete Math.} \textbf{120} (1993), 291--295.

\bibitem{WZ} T. Wang and X. Zhao, Odd graphs and its application on the strong edge coloring. \emph{arXiv preprint} \texttt{arXiv:1412.8358}.

\bibitem{west}
D.~B. West.
\newblock {\em Introduction to graph theory}.
\newblock Prentice Hall, Inc., Upper Saddle River, NJ, 1996.

\bibitem{vu} V. H. Vu, A General Upper Bound on the List Chromatic Number of Locally Sparse Graphs, \textit{Comb. Probab. Comp.}, \textbf{11} (2002), 103--111.


}
\end{thebibliography}
 \end{document}